\documentclass[10pt]{amsart}

\usepackage{amssymb,amsthm,amsmath}
\usepackage[numbers,sort&compress]{natbib}
\usepackage{color}
\usepackage{graphicx}
\usepackage{tikz}

\title[ ]{ANDERSON LOCALIZATION FOR JACOBI MATRICES ASSOCIATED WITH HIGH-DIMENSIONAL SKEW SHIFTS }

\author{Jia Shi}
\address[Jia Shi]{School of Mathematical Sciences,
Fudan University,
Shanghai 200433, China} \email{15110180007@fudan.edu.cn}

\author{Xiaoping Yuan}
\address[Xiaoping Yuan]{School of Mathematical Sciences,
Fudan University,
Shanghai 200433, China} \email{xpyuan@fudan.edu.cn}


\theoremstyle{plain}
\newtheorem{thm}{Theorem}[section]
 \newtheorem{cor}[thm]{Corollary}
 \newtheorem{lem}[thm]{Lemma}
 \newtheorem{prop}[thm]{Proposition}
 \newtheorem{rem}[thm]{Remark}
 
 \numberwithin{equation}{section}

\begin{document}


\begin{abstract}
In this paper, we establish Anderson localization for a
class of Jacobi matrices associated with skew shifts on $\mathbb{T}^{d}$, $d\geq3$.
\end{abstract}

\maketitle
\section{Introduction and main result}

Over the past thirty years, there are many papers on the topic of Anderson localization for
lattice Schr\"odinger operators
\begin{equation}\label{1}
H=v_n\delta_{nn'}+\Delta,
\end{equation}
where $v_n$ is a quasi-periodic potential, $\Delta$ is the lattice Laplacian on $\mathbb{Z}$.
Anderson localization means that $H$ has pure point spectrum with exponentially localized states $\varphi=(\varphi_{n})_{n\in\mathbb{Z}}, $
\begin{equation}\label{2}
|\varphi_{n}|<e^{-c| n|},\quad | n|\rightarrow\infty.
\end{equation}

We may associate the potential $v_n$ to a dynamical system $T$ as follows:
\begin{equation}\label{3}
v_n=\lambda v(T^nx),
\end{equation}
where $v$ is real analytic on $\mathbb{T}^{d}$ and $T$ is a shift on $\mathbb{T}^{d}$:
\begin{equation}\label{4}
Tx=x+\omega.
\end{equation}
Fix $x=x_0$, if $\lambda$ is large and $\omega$ outside set of small measure, $H$ will satisfy Anderson localization.

The proof of Anderson localization is based on multi-scale analysis and semi-algebraic set theory. In this line,
 Bourgain and Goldstein \cite{BG} proved Anderson localization for Schr\"odinger operators (\ref{1}) with  help of  fundamental matrix and Lyapounov exponent.
By multi-scale method, Bourgain, Goldstein and Schlag \cite{BGS2}  proved Anderson localization for Schr\"odinger operators on $\mathbb{Z}^{2}$
\begin{equation}\label{5}
H(\omega_1,\omega_2;\theta_1,\theta_2)=\lambda v(\theta_1+n_1\omega_1,\theta_2+n_2\omega_2)+\Delta.
\end{equation}
Later, Bourgain \cite{B07} proved Anderson localization for quasi-periodic lattice {S}chr\"odinger operators on $\mathbb{Z}^{d}$, $d$ arbitrary.
Recently, using more elaborate semi-algebraic arguments, Bourgain and Kachkovskiy \cite{BK}
proved Anderson localization for two interacting quasi-periodic particles.

More generally, we can study the long range model
\begin{equation}\label{6}
H=v(x+n\omega)\delta_{nn'}+S_\phi,
\end{equation}
with $\Delta$ replaced by a Toeplitz operator
\begin{equation}\label{7}
S_\phi(n,n')=\hat{\phi}(n-n'),
\end{equation}
where $\phi$ is real analytic.
Bourgain \cite{B05} proved Anderson localization for the long-range quasi-periodic operators (\ref{6}).
Note that in this case, we cannot use the fundamental matrix formalism as (\ref{1}).
Bourgain's method in \cite{B05} also permits us to establish Anderson localization for band Schr\"odinger operators \cite{BJ}
\begin{equation}\label{8}
H_{(n,s),(n',s')}(\omega,\theta)=\lambda v_s(\theta+n\omega)\delta_{nn'}\delta_{ss'}+\Delta,
\end{equation}
where $\{v_s|1\leq s\leq b\}$ are real analytic.
Recently, Jian, Shi and Yuan \cite{JSY} proved Anderson localization for quasi-periodic block operators with long-range interactions.

If the transformation $T$ is a skew shift on $\mathbb{T}^{2}$:
\begin{equation}\label{9}
T(x_{1},x_{2})=(x_{1}+x_{2},x_{2}+\omega),
\end{equation}
using transfer matrix and Lyapounov exponent, Bourgain, Goldstein and Schlag \cite{BGS1} proved Anderson localization for
\begin{equation}\label{10}
H=\lambda v(T^nx)+\Delta.
\end{equation}

In order to study quantum kicked rotor equation
\begin{equation}\label{11}
i\frac{\partial\Psi(t,x)}{\partial t}=a\frac{\partial^{2}\Psi(t,x)}{\partial x^{2}}+ib\frac{\partial\Psi(t,x)}{\partial x}+V(t,x)\Psi(t,x),\quad x\in\mathbb{T},
\end{equation}
where
\begin{equation}\label{12}
V(t,x)=\kappa\left(\sum_{n\in\mathbb{Z}}\delta(t-n)\right)\cos(2\pi x),
\end{equation}
using multi-scale method,
Bourgain \cite{B02} proved Anderson localization for the operator
\begin{equation}\label{13}
W=\phi_{m-n}(T^mx),
\end{equation}
where $\phi_k$ are trigonometric polynomials and $T$ is a skew shift on $\mathbb{T}^{2}$.

However, there are few results on high-dimensional skew shifts.
When $d\geq 3$, the skew shift
$T:\mathbb{T}^{d}\rightarrow\mathbb{T}^{d}$ is given by
\begin{equation}
(Tx)_{i}=x_{i}+x_{i+1},\quad 1\leq i\leq d-1,
\end{equation}
\begin{equation}
(Tx)_{d}=x_{d}+\omega,\quad x=(x_1,\ldots, x_d).
\end{equation}
In \cite{K}, Kr\"{u}ger proved positivity of Lyapounov exponents for the Schr\"odinger operator
\begin{equation}
H=\lambda f((T^nx)_{1})\delta_{nn'}+\Delta,
\end{equation}
where $T$ is a skew shift on $\mathbb{T}^{d}$, $d$ is sufficiently large, $f$ is a real, nonconstant function on $\mathbb{T}$.

In this paper, we generalized Bourgain's result on skew shifts on $\mathbb T^2$ \cite{B02} to higher dimensional ones on $\mathbb{T}^{d}, d\geq3$.
More precisely,
we consider matrices $(A_{mn}(x))_{m,n\in\mathbb{Z}}$, $x\in\mathbb{T}^{d}$ associated with a skew shift $T:\mathbb{T}^{d}\rightarrow\mathbb{T}^{d}$ of the form
\begin{equation}\label{a1}
A_{mm}(x)=v(T^{m}x),
\end{equation}
\begin{equation}\label{a2}
A_{mn}(x)=\phi_{m-n}(T^{m}x)+\overline{\phi_{n-m}(T^{n}x)},\quad m\neq n,
\end{equation}
where
\begin{equation}\label{a3}
\mbox{$v$ is a real, nonconstant, trigonometric polynomial},
\end{equation}
\begin{equation}\label{a4}
\mbox{$\phi_{k}$ is a trigonometric polynomial of degree $<|k|^{C_{1}}$},
\end{equation}
\begin{equation}\label{a5}
\|\phi_{k}\|_{\infty}<\gamma e^{-|k|}.
\end{equation}

We will prove the following result:
\begin{thm}
Consider a lattice operator $H_{\omega}(x)$ associated to the skew shift $T=T_{\omega}$ acting on $\mathbb{T}^{d},\ d\geq3$, of the form (\ref{a1})-(\ref{a5}).
Assume $\omega \in DC$ (diophantine condition),
\begin{equation}\label{a6}
\|k\omega\|>c| k|^{-2},\quad \forall k\in\mathbb{Z}\setminus\{0\}.
\end{equation}
Fix $x_0\in\mathbb{T}^{d}$. Then for almost all $\omega \in DC$  and $\gamma$ taken sufficiently small in (\ref{a5}), $H_{\omega}(x_0)$ satisfies Anderson localization.
\end{thm}

We summarize the scheme of the proof.
As mentioned above, the transfer matrix and Lyapounov exponent approach is not applicable to the long range case here.
We will use the multi-scale method developed  in \cite{B02}, \cite{BGS2}. Our basic strategy is the same as that in \cite{B02}, but with more complicated computations.
First, we need Green's function estimates for $G_{[0,N]}(E,x)=(R_{[0,N]}(H(x)-E)R_{[0,N]})^{-1}$, where $R_{\Lambda}$ is the restriction operator to $\Lambda\subset\mathbb{Z}$.
We will prove in Section 3 that
\begin{equation}
   \| G_{[0,N]}(E,x)\| <e^{N^{1-}},
\end{equation}
\begin{equation}
  |G_{[0,N]}(E,x)(m,n)|<e^{-\frac{1}{100}|m-n|},\quad 0\leq m,n\leq N, |m-n|>\frac{N}{10}
\end{equation}
for $x\notin \Omega_{N}(E)$, where
\begin{equation}
{\rm mes} \Omega_{N}(E)<e^{-N^{\sigma}}, \sigma>0.
\end{equation}
The main difficulty here is to study the intersection of $\Omega_{N}(E)$ and skew shift orbits.
We need to prove
\begin{equation} \label{a7}
 \#\{n=1,\ldots,M|T^nx\in\Omega_{N}(E)\}<M^{1-\delta}, \delta>0,
\end{equation}
 where
 \begin{equation}
 \log\log M\ll \log N\ll \log M.
\end{equation}
To obtain (\ref{a7}), we study the ergodic property of skew shifts on $\mathbb{T}^{d}$ in Section 2.

Next,
in Section 4, we use decomposition of semi-algebraic set to estimate
\begin{equation*}
{\rm mes}\left[\omega\in\mathbb{T}\Big\lvert(\omega,T^{j}_{\omega}x)\in A,\ \exists j\sim M\right]<M^{-c},c>0,
\end{equation*}
where $x\in\mathbb{T}^{d}$, $A\subset\mathbb{T}^{d+1}$ is a semi-algebraic set of degree $B$ and measure $\eta$, satisfying
\begin{equation*}
\log B\ll\log M\ll\log\frac{1}{\eta}.
\end{equation*}
This is a key point to eliminate the energy $E$ in the proof of Anderson localization.

Finally,
using Green's function estimates and semi-algebraic set theory, we prove Anderson localization of the operator $H_{\omega}(x)$ in Section 5 as in \cite{BG}, \cite{BGS1}.

We will use the following notations. For positive numbers $a,b,a\lesssim b$ means $Ca\leq b$ for some constant $C>0$.
$a\ll b$ means $C$ is large. $a\sim b$ means $a\lesssim b$ and $b\lesssim a$. $N^{1-}$ means $N^{1-\epsilon}$ with some small $\epsilon>0$.

\section{An ergodic property of skew shifts on $\mathbb{T}^{d}$}

In this section, we prove the following ergodic property of skew shifts on $\mathbb{T}^{d}$.
\begin{lem}\label{l2.1}
Assume $\omega \in DC$, $T=T_{\omega}$ is the skew shift on $\mathbb{T}^{d}$, $\epsilon>L^{-\frac{1}{(d+1)2^{d+1}}}$. Then
\begin{equation*}
\#\{n=1,\ldots,L||T^nx-a\|<\epsilon\}<C\epsilon^{d}L, \quad C=C(d),
\end{equation*}
where $\| x\|=\inf\limits_{m\in\mathbb{Z}}|x-m|,\ x\in\mathbb{T},\ \| x\|=\sum\limits_{i=1}^{d}\| x_{i}\|,\ x=(x_1,\ldots, x_d)\in\mathbb{T}^{d}$.
\end{lem}

\begin{proof}
We assume $a=0$. Let $\chi$ be the indicator function of the ball $B(0,\epsilon)$, $R=\frac{1}{\epsilon}$, $F_{R}$ is the Fejer kernel,
then $\chi\leq C\epsilon^{d}\prod\limits_{j=1}^{d}F_{R}(x_{j})$.

Let $f(x)=\prod\limits_{j=1}^{d}F_{R}(x_{j})$, then
\begin{align*}
\sum_{n=1}^{L}\chi(T^nx)\leq C\epsilon^{d}\sum_{n=1}^{L}f(T^nx)\leq C\epsilon^{d}\sum_{n=1}^{L}\sum_{0\leq|l_j|<R}\hat{f}(l_1,\ldots,l_d)e^{2\pi i\langle T^nx,l\rangle}\\
\leq C\epsilon^{d}\left(L+\sum_{0<\lvert k\mid<\frac{1}{\epsilon}}\Big\lvert\sum_{n=1}^{L}e^{2\pi i\langle T^nx,k\rangle}\Big\lvert\right).
\end{align*}

Let
\begin{equation}\label{2.1}
S_{k}=\Big\lvert\sum_{n=1}^{L}e^{2\pi i\langle T^nx,k\rangle}\Big\lvert,\quad 0<|k|<\frac{1}{\epsilon},
\end{equation}
we only need to prove
\begin{equation}\label{2.2}
\sum_{0<|k|<\frac{1}{\epsilon}}S_{k}\leq CL.
\end{equation}

From the skew shift, we have
\begin{equation}\label{2.3}
(T^nx)_{i}=x_{i}+nx_{i+1}+\cdots+\binom{n}{d-i}x_{d}+\binom{n}{d-i+1}\omega,\quad i=1,\ldots,d,\quad x=(x_1,\ldots, x_d).
\end{equation}

If $k_1=\cdots=k_{d-1}=0$, then
\begin{equation}\label{2.4}
S_{k}=\left|\sum_{n=1}^{L}e^{2\pi ink_d\omega}\right|\leq\frac{1}{\| k_d\omega\|}\leq C| k_d|^{2}.
\end{equation}

If $k_1=\cdots=k_{d-2}=0,k_{d-1}\neq0$, then $S_{k}=\left|\sum\limits_{n=1}^{L}e^{2\pi if(n)}\right|$, where $f(n)=\frac{1}{2}n^{2}k_{d-1}\omega+cn$,
$c$ is independent of $n$.

So,
$$
\begin{aligned}
S_{k}^{2} &= \left(\sum_{n=1}^{L}e^{2\pi if(n)}\right)\left(\sum_{n=1}^{L}e^{-2\pi if(n)}\right) \lesssim L+\sum_{h=1}^{L-1}\left|\sum_{n=1}^{L-h}e^{2\pi i(f(n+h)-f(n))}\right|\\
 &\lesssim L+\sum_{h=1}^{L-1}\min\left(L,\frac{1}{\| hk_{d-1}\omega\|}\right) \lesssim L+\sum_{m=1}^{|k_{d-1}|L}\min\left(L,\frac{1}{\|m\omega\|}\right).
\end{aligned}
$$

Since $\omega \in DC$, we may find an approximant $q$ of $\omega$ satisfying
\begin{equation}\label{2.5}
L^{\frac{1}{2}}<q<L.
\end{equation}

Using
\begin{equation*}
\#\left\{M+1\leq n\leq M+q\Big\lvert\| n\omega-u\|\leq\frac{1}{2q}\right\}\leq3, \quad\forall M\in \mathbb{Z}, u\in \mathbb{R},
\end{equation*}
we get
\begin{equation}\label{2.6}
  \sum_{n=M+1}^{M+q}\min\left(L,\frac{1}{\| n\omega\|}\right)\lesssim L+q\log q.
\end{equation}

By (\ref{2.5}),(\ref{2.6}), we have
\begin{equation*}
S_{k}^{2}\lesssim \frac{|k_{d-1}|L}{q}(L+q\log q)\lesssim|k_{d-1}|L^{\frac{3}{2}}.
\end{equation*}

Hence
\begin{equation}\label{2.7}
S_{k}\leq C|k_{d-1}|^{\frac{1}{2}}L^{\frac{3}{4}}.
\end{equation}

If $k_1=\cdots=k_{d-3}=0,k_{d-2}\neq0$, then $S_{k}=\left|\sum\limits_{n=1}^{L}e^{2\pi ig(n)}\right|$, where
$g(n)=\frac{1}{6}n^{3}k_{d-2}\omega+bn^2+cn$, $b,c$ is independent of $n$.

So,
\begin{equation*}
  S_{k}^{2}\lesssim L+\sum_{h_1=1}^{L-1}\left|\sum_{n=1}^{L-h_1}e^{2\pi ig_{h_{1}}(n)}\right|,\quad g_{h_{1}}(n)=g(n+h_{1})-g(n).
\end{equation*}

We have
$$
\begin{aligned}
 S_{k}^{4} & \lesssim L^{2}+L\sum_{h_1=1}^{L-1}\left|\sum_{n=1}^{L-h_1}e^{2\pi ig_{h_{1}}(n)}\right|^{2}\\
  &\lesssim L^{3}+L\sum_{h_1=1}^{L-1}\sum_{h_2=1}^{L-h_1-1}\left|\sum_{n=1}^{L-h_1-h_2}e^{2\pi i(g_{h_{1}}(n+h_2)-g_{h_{1}}(n))}\right|\\
  &\lesssim L^{3}+L\sum_{h_1=1}^{L}\sum_{h_2=1}^{L}\min\left(L,\frac{1}{\| h_1h_2k_{d-2}\omega\|}\right).
\end{aligned}
$$

Using
\begin{equation*}
  \#\{(h_1,h_2)\in\mathbb{Z}^{2}\mid h_1h_2=N\}\lesssim N^{0+},
\end{equation*}
we get
\begin{equation*}
S_{k}^{4}\lesssim L^{3}+L^{1+}\sum_{m=1}^{|k_{d-2}|L^{2}}\min\left(L,\frac{1}{\| m\omega\|}\right)
\lesssim L^{3}+L^{1+}\frac{|k_{d-2}|L^{2}}{q}(L+q\log q)\lesssim|k_{d-2}|L^{\frac{7}{2}+}.
\end{equation*}

Hence
\begin{equation}\label{2.8}
S_{k}\leq C|k_{d-2}|^{\frac{1}{4}}L^{\frac{7}{8}+}.
\end{equation}

Repeat the argument above, we get
\begin{equation}\label{2.9}
S_{k}\leq C|k_{d-j}|^{\frac{1}{2^{j}}}L^{1-\frac{1}{2^{j+1}}+},\quad k_1=\cdots=k_{d-j-1}=0,k_{d-j}\neq0,\quad 2\leq j \leq d-1 .
\end{equation}

By (\ref{2.4}), (\ref{2.7}), (\ref{2.9}), we have
$$
\begin{aligned}
  \sum_{0<| k|<\frac{1}{\epsilon}}S_{k}&
  \lesssim \sum_{|k_{d}|<\frac{1}{\epsilon}}| k_{d}|^{2}+\frac{1}{\epsilon}\sum_{|k_{d-1}|<\frac{1}{\epsilon}}|k_{d-1}|^{\frac{1}{2}}L^{\frac{3}{4}}
  +\sum_{j=2}^{d-1}\frac{1}{\epsilon^{j}}\left(\sum_{| k_{d-j}|<\frac{1}{\epsilon}}|k_{d-j}|^{\frac{1}{2^{j}}}L^{1-\frac{1}{2^{j+1}}+}\right)\\
  &\lesssim(\frac{1}{\epsilon})^{3}+\frac{1}{\epsilon}(\frac{1}{\epsilon})^{\frac{3}{2}}L^{\frac{3}{4}}+\sum_{j=2}^{d-1}\left((\frac{1}{\epsilon})^{\frac{1}{2^{j}}+j+1}L^{1-\frac{1}{2^{j+1}}+}\right)\lesssim L.
\end{aligned}
$$

This proves (\ref{2.2}) and Lemma \ref{l2.1}.
\end{proof}

\begin{rem}\label{r2.2}
In the proof of Lemma \ref{l2.1}, we only need to assume
\begin{equation*}
  \|k\omega\|>c|k|^{-2},\quad \forall 0<|k|\leq L.
\end{equation*}
\end{rem}

\section{Green's function estimates}

In this section, we will prove the Green's function estimates using multi-scale analysis in \cite{B02}.

We need the following lemma.

\begin{lem}[Lemma 3.16 in \cite{B02}]\label{l3.1}
Let $A(x)=\{A_{mn}(x)\}_{1\leq m,n\leq N}$ be a matrix valued function on $\mathbb{T}^{d}$ such that
\begin{equation}\label{b1}
\mbox{$A(x)$ is self adjoint for $x\in\mathbb{T}^{d}$},
\end{equation}
\begin{equation}\label{b2}
\mbox{$A_{mn}(x)$ is a trigonometric polynomial of degree $<N^{C_{1}}$},
\end{equation}
\begin{equation}\label{b3}
|A_{mn}(x)|<C_2e^{-c_2|m-n|},
\end{equation}
where $c_2,C_1,C_2>0$ are constants.

Let $0<\delta<1$ be sufficiently small, $M=N^{\delta^{6}},\ L_0=N^{\frac{1}{100}\delta^{2}},\ 0<c_3<\frac{1}{10}c_2.$

Assume that for any interval $I\subset[1,N]$ of size $L_0$, except for $x$ in a set of measure at most $e^{-L_0^{\delta^{3}}}$,
\begin{equation}\label{b4}
\|(R_IA(x)R_I)^{-1}\|<e^{L_0^{1-}},
\end{equation}
\begin{equation}\label{b5}
|(R_IA(x)R_I)^{-1}(m,n)|<e^{-c_3|m-n|},\quad m,n\in I,|m-n|>\frac{L_0}{10}.
\end{equation}

Fix $x\in\mathbb{T}^{d}, n_0\in[1,N]$ is called a good site if $I_0=[n_0-\frac{M}{2},n_0+\frac{M}{2}]\subset[1,N]$,
\begin{equation}\label{b6}
\|( R_{I_{0}}A(x)R_{I_{0}})^{-1}\|<e^{M^{1-}},
\end{equation}
\begin{equation}\label{b7}
| (R_{I_{0}}A(x)R_{I_{0}})^{-1}(m,n)|<e^{-c_3|m-n|},\quad m,n\in I_0,|m-n|>\frac{M}{10}.
\end{equation}

Denote $\Omega(x)\subset[1,N]$ the set of bad sites.
Assume that for any interval $J\subset[1,N], |J|>N^{\frac{\delta}{5}}$, we have $|J\cap\Omega(x)|<|J|^{1-\delta}$.

Then
\begin{equation}\label{b8}
\| A(x)^{-1}\|<e^{N^{1-\frac{\delta}{C(d)}}},
\end{equation}
\begin{equation}\label{b9}
| A(x)^{-1}(m,n)|<e^{-c'_3|m-n|},\quad m,n\in [1,N],|m-n|>\frac{N}{10}
\end{equation}
except for $x$ in a set of measure at most $e^{-\frac{N^{\delta^{2}}}{C(d)}}$, where $C(d)$ is a constant depending on $d$, $c_{3}'>c_{3}-(\log N)^{-8}$.
\end{lem}

By Lemma \ref{l2.1}, Lemma \ref{l3.1}, we can prove the Green's function estimates.

\begin{prop}\label{p3.2}
  Let $T=T_{\omega}:\mathbb{T}^{d}\rightarrow\mathbb{T}^{d}$ be the skew shift with frequency $\omega$ satisfying
  \begin{equation}\label{b10}
   \| k\omega\|>c|k|^{-2}, \quad \forall 0<|k|\leq N.
   \end{equation}
   $A_{mn}(x)$ is the form (\ref{a1})-(\ref{a5}),$\gamma$ in (\ref{a5}) is small.

Then for all $N$ and energy $E$,
 \begin{equation}\label{b11}
   \| G_{[0,N]}(E,x)\| <e^{N^{1-}},
 \end{equation}
  \begin{equation}\label{b12}
  |G_{[0,N]}(E,x)(m,n)|<e^{-\frac{1}{100}|m-n|},\quad  0\leq m,n\leq N, |m-n|>\frac{N}{10}
  \end{equation}
 for $x\notin \Omega_{N}(E)$, where
 \begin{equation}\label{b13}
 {\rm mes} \Omega_{N}(E)<e^{-N^{\sigma}},\sigma>0.
 \end{equation}
\end{prop}

\begin{proof}
  Since $T^n(x_1,\ldots,x_d)=\left(x_1+nx_2+\cdots+\binom{n}{d-1}x_d+\binom{n}{d}\omega,\ldots,x_d+n\omega\right)$, $A_{mn}(x)$ is a trigonometric polynomial in $x$ of degree $<(|m|+|n|)^{C_1+d}$,
  $\{A_{mn}(x)-E\}_{0\leq m,n\leq N}$ satisfy (\ref{b1})-(\ref{b3}) with $c_2=1,C_2=\gamma$.

  First fix any large initial scale $N_0$ and choose $\gamma=\gamma(N_0)$ small, using Lojasiewicz's inequality (see Section 4 in \cite{B02}),
  we get
  \begin{equation}\label{b14}
  |G_{[0,N_0]}(E,x)(m,n)|<e^{N_0^{\frac{1}{2}}-\frac{1}{2}|m-n|},\quad 0\leq m,n\leq N_0
  \end{equation}
  except for $x$ in a set of measure $<e^{-cN_0^{\frac{1}{2}}}$.

  Then we estabish inductively on the scale $N$ that
   \begin{equation}\label{b15}
   {\rm mes}\left[x\in\mathbb{T}^{d}\Big\lvert|G_{[0,N]}(E,x)(m,n)|>e^{N^{1-}-c_3|m-n|\chi_{|m-n|>\frac{N}{10}}} ,\ \exists 0\leq m,n\leq N\right]<e^{-N^{\delta^{3}}},
  \end{equation}
  where $c_3>\frac{1}{100},\ 0<\delta<1$ is a fixed small number.

(\ref{b14}) implies (\ref{b15}) for an initial large scale $N_0$.

Assume (\ref{b15}) holds up to scale $L_0=N^{\frac{1}{100}\delta^{2}}$.
Since $A_{m+1,n+1}(x)=A_{mn}(Tx)$, we have
\begin{equation*}
    R_{I}(A(x)-E)R_{I}= R_{[0,N]}(A(T^{n}x)-E)R_{[0,N]},G_{I}(E,x)=G_{[0,N]}(E,T^{n}x),\quad I=[n,n+N].
  \end{equation*}
  Since $T$ is measure preserving,(\ref{b4}),(\ref{b5}) will hold for $x$ outside a set of measure at most $e^{-L_0^{\delta^{3}}}$.
  Denote $\Omega(x)\subset[0,N]$ the set of bad sites with respect to scale $M=N^{\delta^{6}}$. $n_0\notin\Omega(x)$ means

  \begin{align} \label{b16}
  &|G_{[0,M]}(E,T^{n_{0}-\frac{M}{2}}x)(m,n)|= \\
\nonumber  &|G_{[n_{0}-\frac{M}{2},n_{0}+\frac{M}{2}]}(E,x)(m+n_{0}-\frac{M}{2},n+n_{0}-\frac{M}{2})|<e^{M^{1-}-c_3|m-n|\chi_{|m-n|>\frac{M}{10}}}.
  \end{align}

  From the inductive hypothesis, we have
  \begin{align}\label{b17}
  |G_{[0,M]}(E,x)(m,n)|<e^{M^{1-}-c_3|m-n|\chi_{|m-n|>\frac{M}{10}}}, 0\leq m,n\leq M,\quad \forall x\notin\Omega,\quad {\rm mes}\Omega<e^{-M^{\delta^{3}}}.
  \end{align}

  By (\ref{b16}), (\ref{b17}), Lemma \ref{l3.1}, we only need to show that for any $x\in\mathbb{T}^{d},\ N^{\frac{\delta}{5}}<L<N$,
   \begin{align}\label{b18}
  \#\{1\leq n\leq L|T^nx\in\Omega\}<L^{1-\delta}.
  \end{align}

  Since $A_{mn}(x)$ is a trigonometric polynomial of degree $<(|m|+|n|)^{C}$, we can express $G_{[0,M]}(E,x)(m,n)$ as a ratio of determinants to write
(\ref{b17}) in the form
 \begin{align}\label{b19}
  P_{mn}(\cos x_1,\sin x_1,\ldots,\cos x_d,\sin x_d)\leq 0,
  \end{align}
 where $P_{mn}$  is a polynomial of degree at most $M^{C}$. Replacing $\cos, \sin$ by truncated power series, permits us to replace (\ref{b19}) by
  \begin{align}\label{b20}
   P_{mn}( x_1,\ldots, x_d)\leq 0,\quad {\rm deg} P_{mn}<M^{C}.
  \end{align}
So, $\Omega$ may be viewed as a semi-algebraic set of degree at most $M^{C}$.
(For properties of semi-algebraic sets, see Section 4.)
When $\epsilon>e^{-\frac{1}{d}M^{\delta^{3}}}$, by Corollary \ref{c4.4}, $\Omega$ may be covered by at most $M^{C}(\frac{1}{\epsilon})^{d-1} \epsilon$-balls.
Choosing $\epsilon=L^{-\frac{1}{(d+1)2^{d+1}}}>N^{-1}>e^{-\frac{1}{d}M^{\delta^{3}}}$, by (\ref{b10}), using Lemma \ref{l2.1}, Remark \ref{r2.2}, we have
\begin{equation*}
  \#\{1\leq n\leq L|T^nx\in\Omega\}<M^{C}(\frac{1}{\epsilon})^{d-1}\epsilon^{d}L<L^{C\delta^{5}+1-\frac{1}{(d+1)2^{d+1}}}<L^{1-\delta},
\end{equation*}
when $\delta$ is small enough.

This proves (\ref{b18}) and Proposition \ref{p3.2}.
\end{proof}

\section{Semi-algebraic sets}

We recall some basic facts of semi-algebraic sets. Let $\mathcal{P}=\{P_1,\ldots,P_s\}\subset\mathbb{R}[X_1,\ldots,X_n]$
be a family of real polynomials whose degrees are bounded by $d$.
A semi-algebraic set is given by
\begin{equation}\label{c1}
S=\bigcup_{j}\bigcap_{l\in L_{j}}\left\{\mathbb{R}^{n}\Big\lvert P_ls_{jl}0\right\},
\end{equation}
where $L_{j}\subset\{1,\ldots,s\},s_{jl}\in\{\leq,\geq,=\}$ are arbitrary.
We say that $S$ has degree at most $sd$ and its degree is the $\inf$ of $sd$ over all representations as in (\ref{c1}).

The projection of a semi-algebraic set of $\mathbb{R}^{n}$ onto $\mathbb{R}^{m}$ is semi-algebraic.
\begin{prop}[\cite{BPR}]\label{p4.1}
  Let $S\subset\mathbb{R}^{n}$ be a semi-algebraic set of degree $B$, then any projection of $S$ has degree at most $B^{C}, C=C(n)$.
\end{prop}

We need the following bound on the number of connected components.
\begin{prop}[\cite{B}]\label{p4.2}
Let $S\subset\mathbb{R}^{n}$ be a semi-algebraic set of degree $B$, then the number of connected components of $S$ is bounded by $B^{C}, C=C(n)$.
\end{prop}

A more advanced part of the theory of semi-algebraic sets is the following triangulation theorem.
\begin{thm}[\cite{G}]\label{t4.3}
For any positive integers $r,n$, there exists a constant $C=C(n,r)$ with the following property:
Any semi-algebraic set $S\subset[0,1]^{n}$ can be triangulated into $N\lesssim({\rm deg} S+1)^{C}$ simplices,
where for every closed $k$-simplex $\Delta\subset S$, there exists a homeomorphism $h_{\Delta}$ of the regular simplex $\Delta^{k}\subset\mathbb{R}^{k}$ with unit edge length onto
$\Delta$ such that $\|D_rh_\Delta\|\leq1$.
\end{thm}

\begin{cor}[Corollary 9.6 in \cite{B05}]\label{c4.4}
 Let $S\subset[0,1]^{n}$ be semi-algebraic of degree $B$.
 Let $\epsilon>0, \ {\rm mes}_nS<\epsilon^{n}$. Then $S$ may be covered by at most $B^{C}(\frac{1}{\epsilon})^{n-1} \epsilon$-balls.
\end{cor}

Finally, we will make essential use of the following transversality property.

\begin{lem}[(1.5) in \cite{B07}]\label{l4.5}
 Let $S\subset[0,1]^{n=n_1+n_2}$ be a semi-algebraic set of degree $B$ and
 \begin{equation}\label{c2}
{\rm mes}_{n}S<\eta, \quad \log B\ll\log\frac{1}{\eta},\quad \epsilon>\eta^{\frac{1}{n}},
 \end{equation}
 denote $(x,y)\in[0,1]^{n_{1}}\times[0,1]^{n_{2}}$ the product variable.

 Then there is a decomposition $S=S_1\cup S_2$,
 $S_1$ satisfying
 \begin{equation}\label{c3}
{\rm mes}_{n_1}({\rm Proj}_x S_1)<B^{C}\epsilon
 \end{equation}
 and $S_2$ satisfying the transversality property
\begin{equation}\label{c4}
 {\rm mes}_{n_2}(S_2\cap L)<B^{C}\epsilon^{-1}\eta^{\frac{1}{n}}
\end{equation}
for any $n_2$-dimensional hyperplane $L$ such that $\max\limits_{1\leq j\leq n_1}|{\rm Proj}_L(e_j)|<\frac{\epsilon}{100}$, where $\{e_j|1\leq j\leq n_1\}$ are $x$-coordinate vectors.
\end{lem}

Now we can prove the following lemma.
\begin{lem}\label{l4.6}
 Let $S\subset[0,1]^{d+1}$ be a semi-algebraic set of degree $B$ such that
 \begin{equation}\label{c5}
{\rm mes}S<e^{-B^{\sigma}}, \sigma>0.
 \end{equation}

Let $M$ satisfy
\begin{equation}\label{c6}
\log\log M\ll\log B\ll\log M.
\end{equation}

Then for all $x\in \mathbb{T}^{d}$,
\begin{equation}\label{c7}
{\rm mes}\left[\omega\in\mathbb{T}\Big\lvert(\omega,T^{j}_{\omega}x)\in S,\ \exists j\sim M\right]<M^{-c},c>0.
\end{equation}
\end{lem}

\begin{proof}

For $x^{0}=(x_1^{0},\ldots,x_d^{0})\in \mathbb{T}^{d}$, we study the intersection of $S\subset[0,1]^{d+1}$ and sets
\begin{equation}\label{c8}
\{(\omega,x_1,\ldots,x_d)|\omega\in[0,1]\},
\end{equation}
where $x_i=(T_{\omega}^{j}x^{0})_{i}=x_{i}^{0}+jx_{i+1}^{0}+\cdots+\binom{j}{d-i}x_{d}^{0}+\binom{j}{d-i+1}\omega$, $1\leq i\leq d$ are considered (mod 1).

By (\ref{c5}), (\ref{c6}), we have
\begin{equation}\label{c9}
{\rm mes}_{d+1}S<\eta=e^{-B^{\sigma}}, \quad\log B\ll\log M\ll\log\frac{1}{\eta}.
\end{equation}

Take $\epsilon=M^{-1+}$ and apply Lemma \ref{l4.5}, $S=S_1\cup S_2$.
Since ${\rm mes}_{1}({\rm Proj}_\omega S_1)<B^{C}M^{-1+}=M^{-1+}$, restriction of $\omega$ permits us to replace $S$ by $S_2$ satisfying
\begin{equation}\label{c10}
{\rm mes}_d(S_2\cap L)<B^{C}\epsilon^{-1}\eta^{\frac{1}{d+1}}<\eta^{\frac{1}{d+2}},
\end{equation}
whenever $L$ is a $d$-dimensional hyperplane satisfying $|{\rm Proj}_L(e_0)|<\frac{\epsilon}{100}$, where $e_0$ is the $\omega$-coordinate vector.

Fixing $j$,(\ref{c8}) considered as subset of $[0,1]^{d+1}$ lies in the union of the parallel $d$-dimensional hyperplanes
\begin{equation}\label{c11}
Q_{m_1}^{(j)}=\left[\omega=\frac{x_d}{j}\right]-\frac{m_{1}+x_d^{0}}{j}e_0,\quad |m_{1}|<M.
\end{equation}

By (\ref{c10}), we have
\begin{equation}\label{c12}
{\rm mes}_d(S\cap Q_{m_{1}})<\eta^{\frac{1}{d+2}}.
\end{equation}

Fixing $m_{1}$, consider the semi-algebraic set $S\cap Q_{m_{1}}$ and its intersection with the parallel $(d-1)$-dimensional hyperplanes
\begin{equation}\label{c13}
Q_{m_{1},m_{2}}^{(j)}=Q_{m_{1}}\cap\left[x_d=\frac{2}{j-1}x_{d-1}-\frac{2}{j-1}\left(x_{d-1}^{0}+\frac{j+1}{2}x_d^{0}+m_{2}\right)\right],\quad |m_{2}|<M.
\end{equation}

Take $\epsilon=M^{-1+}$ and apply Lemma \ref{l4.5} in $Q_{m_{1}}$, $S\cap Q_{m_{1}}=S_{m_{1}}^{1}\cup S_{m_{1}}^{2}$,
where
\begin{equation}\label{c14}
\mbox{${\rm Proj}_{x_{d}}S_{m_{1}}^{1}$ is a union of at most $B^{C}$ intervals of measure at most $B^{C}M^{-1+}$,}
\end{equation}
and by (\ref{c12}), we have
\begin{equation}\label{c15}
{\rm mes}_{d-1}(S_{m_{1}}^{2}\cap Q_{m_{1},m_{2}})<B^{C}M\eta^{\frac{1}{d(d+2)}}<\eta^{\frac{1}{(d+2)^{2}}}.
\end{equation}

Fixing $m_{2}$, consider the semi-algebraic set $S_{m_{1}}^{2}\cap Q_{m_{1},m_{2}}$ and its intersection with the parallel $(d-2)$-dimensional hyperplanes
\begin{equation*}
Q_{m_{1},m_{2},m_{3}}^{(j)}=Q_{m_{1},m_{2}}\cap\left[x_{d-1}=\frac{3}{j-2}x_{d-2}-\frac{3}{j-2}\left(x_{d-2}^{0}+\cdots+\frac{j(j+1)}{6}x_{d}^{0}+m_{3}\right)\right],
\end{equation*}
where $|m_{3}|<M$.

Take $\epsilon=M^{-1+}$ and apply Lemma \ref{l4.5} in $Q_{m_{1},m_{2}}$, $S_{m_{1}}^{2}\cap Q_{m_{1},m_{2}}=S_{m_{1},m_{2}}^{1}\cup S_{m_{1},m_{2}}^{2}$,
where
\begin{equation*}
\mbox{${\rm Proj}_{x_{d-1}}S_{m_{1},m_{2}}^{1}$ is a union of at most $B^{C}$ intervals of measure at most $B^{C}M^{-1+}$,}
\end{equation*}
and by (\ref{c15}), we have
\begin{equation*}
{\rm mes}_{d-2}(S_{m_{1},m_{2}}^{2}\cap Q_{m_{1},m_{2},m_{3}})<\eta^{\frac{1}{(d+2)^{3}}}.
\end{equation*}

Repeat the argument above, fixing $m_{i}, 2\leq i\leq d-1$, consider the semi-algebraic set $S_{m_{1},\ldots,m_{i-1}}^{2}\cap Q_{m_{1},\ldots,m_{i}}$
and its intersection with the parallel $(d-i)$-dimensional hyperplanes
\begin{equation}\label{c16}
Q_{m_{1},\ldots,m_{i+1}}^{(j)}=Q_{m_{1},\ldots,m_{i}}\cap\left[x_{d-i+1}=\frac{i+1}{j-i}x_{d-i}-\frac{i+1}{j-i}\left(x_{d-i}^{0}+\cdots+\frac{1}{i+1}\binom{j+1}{i}x_{d}^{0}+m_{i+1}\right)\right],
\end{equation}
where $|m_{i+1}|<M$.

Take $\epsilon=M^{-1+}$ and apply Lemma \ref{l4.5} in $Q_{m_{1},\ldots,m_{i}}$, 
$S_{m_{1},\ldots,m_{i-1}}^{2}\cap Q_{m_{1},\ldots,m_{i}}=S_{m_{1},\ldots,m_{i}}^{1}\cup S_{m_{1},\ldots,m_{i}}^{2}$,
where
\begin{equation}\label{c17}
\mbox{${\rm Proj}_{x_{d-i+1}}S_{m_{1},\ldots,m_{i}}^{1}$ is a union of at most $B^{C}$ intervals of measure at most $B^{C}M^{-1+}$,}
\end{equation}
and
\begin{equation}\label{c18}
{\rm mes}_{d-i}(S_{m_{1},\ldots,m_{i}}^{2}\cap Q_{m_{1},\ldots,m_{i+1}})<\eta^{\frac{1}{(d+2)^{i+1}}}.
\end{equation}

Finally, fixing $m_{d-1}$, consider the semi-algebraic set $S_{m_{1},\ldots,m_{d-2}}^{2}\cap Q_{m_{1},\ldots,m_{d-1}}$ and its intersection with the parallel lines
\begin{equation}\label{c19}
Q_{m_{1},\ldots,m_{d}}^{(j)}=Q_{m_{1},\ldots,m_{d-1}}\cap\left[x_{2}=\frac{d}{j-d+1}x_{1}-\frac{d}{j-d+1}\left(x_{1}^{0}+\cdots+\frac{1}{d}\binom{j+1}{d-1}x_{d}^{0}+m_{d}\right)\right],
\end{equation}
where $|m_{d}|<M$.

Take $\epsilon=M^{-1+}$ and apply Lemma \ref{l4.5} in $Q_{m_{1},\ldots,m_{d-1}}$,
$S_{m_{1},\ldots,m_{d-2}}^{2}\cap Q_{m_{1},\ldots,m_{d-1}}=S_{m_{1},\ldots,m_{d-1}}^{1}\cup S_{m_{1},\ldots,m_{d-1}}^{2}$,
where
\begin{equation}\label{c20}
\mbox{${\rm Proj}_{x_{2}}S_{m_{1},\ldots,m_{d-1}}^{1}$ is a union of at most $B^{C}$ intervals of measure at most $B^{C}M^{-1+}$,}
\end{equation}
and
\begin{equation}\label{c21}
{\rm mes}_{1}(S_{m_{1},\ldots,m_{d-1}}^{2}\cap Q_{m_{1},\ldots,m_{d}})<\eta^{\frac{1}{(d+2)^{d}}}.
\end{equation}

Summing (\ref{c21}) over $j,m_{1},\ldots,m_{d}$, the collected contribution in the $\omega$-parameter is less than $M^{-d}M^{d+1}B^{C}M\eta^{\frac{1}{(d+2)^{d}}}<\eta^{\frac{1}{(d+2)^{d+1}}}$.
So, we only need to consider the contribution of $S_{m_{1},\ldots,m_{i}}^{1}$ (\ref{c17}).
We just deal with $S_{m_{1},\ldots,m_{d-1}}^{1}$ below, since for other sets, the method is similar.

If (\ref{c7}) fails, we have
\begin{equation*}
\sum_{j\sim M,|m_{1}|,\ldots,|m_{d}|<M}{\rm mes} \left[{\rm Proj}_\omega {\rm Proj}_{x_2}(S_{m_{1},\ldots,m_{d-1}}^{1}\cap Q_{m_{1},\ldots,m_{d}}^{(j)})\right]>M^{0-},
\end{equation*}
\begin{equation}\label{c22}
\sum_{j\sim M,|m_{1}|,\ldots,|m_{d}|<M}{\rm mes} \left[{\rm Proj}_{x_2}(S_{m_{1},\ldots,m_{d-1}}^{1}\cap Q_{m_{1},\ldots,m_{d}}^{(j)})\right]>M^{d-1-}.
\end{equation}
So, there is a set $J\subset\mathbb{Z}\cap[j\sim M], |J|>M^{1-}$ such that for each $j\in J$, there are at least $M^{d-1-}$ values of $(m_{1},\ldots,m_{d-1})$ satisfying
\begin{equation}\label{c23}
\sum_{|m_{d}|<M}{\rm mes} \left[{\rm Proj}_{x_2}(S_{m_{1},\ldots,m_{d-1}}^{1}\cap Q_{m_{1},\ldots,m_{d}}^{(j)})\right]>M^{-1}.
\end{equation}

By (\ref{c20}), $S_{m_{1},\ldots,m_{d-1}}^{1}\cap Q_{m_{1},\ldots,m_{d}}^{(j)}\neq\emptyset$ for at most $M^{0+}$ values of $m_{d}$. Hence
\begin{equation}\label{c24}
\max_{m_{d}}{\rm mes}_{1}(S\cap Q_{m_{1},\ldots,m_{d}}^{(j)})>M^{0-}.
\end{equation}

For fixed $j$,
\begin{equation}\label{c25}
Q_{m_{1},\ldots,m_{d}}^{(j)}//\xi_{j}//\left(1,\binom{j}{d},\ldots,j\right)^{T},\quad \|\xi_{j}\|=1.
\end{equation}

Denote $S_x$ the intersection of $S$ and the $d$-dimensional hyperplane $[x'=x]$.
From (\ref{c24}), to each $(m_{1},\ldots,m_{d-1})$ we can associate some $m_{d}$, such that
\begin{equation}\label{c26}
\int_{0}^{1}\#\{|m_{1}|,\ldots,|m_{d-1}|<M| S_{x}\cap Q_{m_{1},\ldots,m_{d}}\neq\emptyset\}dx>M^{d-1-}.
\end{equation}

If ${\rm mes}_dS_{x}<\eta^{\frac{1}{2}}$, then $S_{x}\cap Q_{m_{1},\ldots,m_{d}}\neq\emptyset$ implies ${\rm dist}(Q_{m_{1},\ldots,m_{d}},\partial S_{x})<\eta^{\frac{1}{2d}}$,
where $\partial S_{x}$ is a union of at most $B^{C}$ connected $(d-1)$-dimensional algebraic set of degree at most $B^{C}$.
From (\ref{c26}), it follows that there is a fixed $(d-1)$-dimensional algebraic set $\Gamma=\Gamma^{(j)}$ of degree at most $B^{C}$ such that
for $x\in[0,1]$ in a set of measure $>M^{0-}$, there are at least $M^{d-1-} \frac{1}{M}$-separated points that are $\eta^{\frac{1}{2d}}$-close
to both $\partial S_{x}$ and $\Gamma+x\xi_{j}$. Hence $(\Gamma+x\xi_{j})\cap S_{\eta_1}$($\eta_1$-neighborhood of $S,\eta_1=2\eta^{\frac{1}{2d}}$)
contains at least $M^{d-1-} \frac{1}{M}$-separated points. So, ${\rm mes}_{d-1}((\Gamma+x\xi_{j})\cap S_{\eta_1})>M^{0-}$.

The hypercylinder $\mathcal{C}^{(j)}=t\xi_{j}+\Gamma^{(j)}$ satisfies
\begin{equation}\label{c27}
{\rm mes}_{d}(\mathcal{C}^{(j)}\cap S_{\eta_1})>M^{0-}.
\end{equation}

By Corollary \ref{c4.4}, we have
\begin{equation}\label{c28}
{\rm mes}_{d+1}S_{\eta_1}<B^{C}\eta_1.
\end{equation}

Since (\ref{c27}) holds for all $j\in J$, by (\ref{c27}), (\ref{c28}), we have
\begin{equation*}
\sum_{j_{1},\ldots,j_{d+1}\in J}{\rm mes}_{d+1}[\bigcap_{1\leq i\leq d+1}\mathcal{C}_{\eta_1}^{(j_{i})}]>\eta_1M^{d+1-}.
\end{equation*}

So, there are distinct $j_{1},\ldots,j_{d+1}\sim M$ such that
\begin{equation}\label{c29}
{\rm mes}_{d+1}[\bigcap_{1\leq i\leq d+1}\mathcal{C}_{\eta_1}^{(j_{i})}]>\eta_1M^{0-}.
\end{equation}

By (\ref{c25}), using Vandermonde determinant, we have
\begin{equation}\label{c30}
\det [\xi_{j_{1}},\ldots,\xi_{j_{d+1}}]\neq 0,
\end{equation}
for distinct $j_{1},\ldots,j_{d+1}$.
So, the vectors $\xi_{j_{1}},\ldots,\xi_{j_{d+1}}$ are not in any $d$-dimensional hyperplane.
Since $\log M\ll \log\frac{1}{\eta_{1}}$, this leads to a contradiction with (\ref{c29}).

This proves Lemma \ref{l4.6}.
\end{proof}

\section{Proof of Anderson localization}

In this section, we give the proof of Anderson localization as in \cite{BG}.

By application of the resolvent identity, we have the following
\begin{lem}[Lemma 10.33 in \cite{B05}]\label{l5.1}
Let $I\subset\mathbb{Z}$ be an interval of size $N$ and $\{I_{\alpha}\}$ be subintervals of size $M\ll N$. Assume that
\begin{itemize}
\item [(i)] If $k\in I$, then there is some $\alpha$ such that $[k-\frac{M}{4},k+\frac{M}{4}]\cap I\subset I_\alpha$.

\item [(ii)] For all $\alpha$,
\begin{equation*}
\|G_{I_{\alpha}}\|<e^{M^{1-}},
  |G_{I_{\alpha}}(n_1,n_2)|<e^{-c_0|n_1-n_2|},   \ n_1,n_2\in I_{\alpha},|n_1-n_2|>\frac{M}{10}.
\end{equation*}
\end{itemize}
Then
\begin{equation*}
|G_{I}(n_1,n_2)|<e^{-(c_0-)|n_1-n_2|}, \  n_1,n_2\in I,|n_1-n_2|>\frac{N}{10}.
\end{equation*}
\end{lem}

Let $T=T_{\omega}$ be the skew shift on $\mathbb{T}^{d}$ with frequency $\omega$ satisfying
\begin{equation}\label{d1}
\| k\omega\|>c| k|^{-2},\quad \forall k\in\mathbb{Z}\setminus\{0\}.
\end{equation}
Fix $x_0\in\mathbb{T}^{d}$.
\begin{equation}\label{d2}
H(x_0)(m,m)=v(T^{m}x_0),
\end{equation}
\begin{equation}\label{d3}
H(x_0)(m,n)=\phi_{m-n}(T^{m}x_0)+\overline{\phi_{n-m}(T^{n}x_0)},\quad m\neq n
\end{equation}
with $v$ and $\phi_{k}$ satisfying (\ref{a3})-(\ref{a5}) and $\gamma$ taken small enough.
Then we have
\begin{thm}\label{t5.2}
For almost all $\omega$ satisfying (\ref{d1}), the lattice operator $H_\omega(x_0)$ satisfies Anderson localization.
\end{thm}

\begin{proof}
To establish Anderson localization, it suffices to show that if $\xi=(\xi_n)_{n\in\mathbb{Z}},E\in\mathbb{R}$ satisfy
\begin{equation}\label{d4}
|\xi_n|<C|n|,\quad |n|\rightarrow\infty,
\end{equation}
\begin{equation}\label{d5}
H(x_0)\xi=E\xi,
\end{equation}
then
\begin{equation}\label{d6}
|\xi_n|<e^{-c|n|},\quad |n|\rightarrow\infty.
\end{equation}

Let $M=N^{C_0}, L=M^{C}$. Denote $\Omega\subset\mathbb{T}^{d}$ the set of $x$ such that
\begin{equation*}
  |G_{[-M,M]}(E,x)(m,n)|<e^{M^{1-}-\frac{1}{100}|m-n|\chi_{|m-n|>\frac{M}{10}}}
\end{equation*}
fails for some $|m|,|n|\leq M$. It was shown in Section 3 that
\begin{equation*}
  \#\{1\leq|n|\leq L|T^nx_0\in\Omega\}<L^{1-\delta}.
\end{equation*}
So, we may find an interval $I\subset[0,L]$ of size $M$, such that
\begin{equation*}
T^{n_{0}}x_{0}\notin\Omega,\quad \forall n_{0}\in I\cup(-I).
\end{equation*}
Hence
\begin{equation}\label{d7}
|G_{[n_{0}-M,n_{0}+M]}(E,x_{0})(m,n)|<e^{M^{1-}-\frac{1}{100}|m-n|\chi_{|m-n|>\frac{M}{10}}}, m,n\in [n_{0}-M,n_{0}+M].
\end{equation}
By (\ref{d4}), (\ref{d5}), (\ref{d7}), we have
\begin{equation}\label{d8}
|\xi_{n_0}|\leq\sum_{n'\in [n_{0}-M,n_{0}+M],n''\notin [n_{0}-M,n_{0}+M]}e^{M^{1-}-\frac{1}{100}|n_0-n'|\chi_{|n_0-n'|>\frac{M}{10}}}e^{-|n'-n''|}|\xi_{n''}|<e^{-\frac{M}{200}}.
\end{equation}

Denoting $j_0$ the center of $I$, we have
\begin{equation}\label{d9}
1=|\xi_{0}|\leq\|G_{[-j_0,j_0]}(x_0,E)\|\|R_{[-j_0,j_0]}H(x_0)R_{\mathbb{Z}\setminus[-j_0,j_0]}\xi\|.
\end{equation}
By (\ref{d4}), (\ref{d8}), we have for $|n|\leq j_0$,
\begin{equation}\label{d10}
|(H(x_0)R_{\mathbb{Z}\setminus[-j_0,j_0]}\xi)_{n}|\leq\sum_{|n'|>j_0}e^{-|n-n'|}|\xi_{n'}|<e^{-\frac{M}{400}}+\sum_{|n'|>j_0+\frac{M}{2}}e^{-|n-n'|}|\xi_{n'}|<e^{-\frac{M}{500}}.
\end{equation}
By (\ref{d9}), (\ref{d10}), we have
\begin{equation}\label{d11}
\|G_{[-j_0,j_0]}(x_0,E)\|>e^N.
\end{equation}

So if there is an extended state $\xi,\xi_{0}=1$ with energy $E$, then there is some $j_0, |j_0|<N_1=N^{C_1}$ ($C_1$ is a sufficiently large constant),
such that (by (\ref{d11}))
\begin{equation}\label{d12}
{\rm dist}(E, {\rm spec} H_{[-j_0,j_0]}(x_0))<e^{-N}.
\end{equation}

Denote $\Omega(E)\subset\mathbb{T}^{d}$ the set of $x$ such that
\begin{equation*}
  |G_{[-N,N]}(E,x)(m,n)|<e^{N^{1-}-\frac{1}{100}|m-n|\chi_{|m-n|>\frac{N}{10}}}
\end{equation*}
fails for some $|m|,|n|\leq N$. Let $\mathcal{E}_{\omega}=\bigcup\limits_{| j|\leq N_1}{\rm spec}H_{[-j,j]}(x_0)$.
It follows from (\ref{d12}) that if $x\notin\bigcup\limits_{E'\in\mathcal{E}_{\omega}}\Omega(E')$, then
\begin{equation}\label{d13}
  |G_{[-N,N]}(E,x)(m,n)|<e^{N^{1-}-\frac{1}{100}|m-n|\chi_{|m-n|>\frac{N}{10}}} ,\quad |m|,|n|\leq N.
\end{equation}

Consider the set $S=S_N\subset\mathbb{T}^{d+1}\times\mathbb{R}$ of $(\omega,x,E')$, where
\begin{equation}\label{d14}
 \| k\omega\|>c|k|^{-2},\quad \forall 0<|k|\leq N,
\end{equation}
\begin{equation}\label{d15}
x\in\Omega(E'),
\end{equation}
\begin{equation}\label{d16}
E'\in\mathcal{E}_{\omega}.
\end{equation}

By (\ref{d14}), (\ref{d15}), (\ref{d16}),
\begin{equation}\label{d17}
\mbox{${\rm Proj}_{\mathbb{T}^{d+1}}S$ is a semi-algebraic set of degree $<N^C$},
\end{equation}
and by Proposition \ref{p3.2},
\begin{equation}\label{d18}
{\rm mes}({\rm Proj}_{\mathbb{T}^{d+1}}S)<e^{-\frac{1}{2}N^{\sigma}}.
\end{equation}

Let $N_2=e^{(\log N)^{2}}$,
\begin{equation}\label{d19}
\mathcal{B}_{N}=\{\omega\in\mathbb{T}\mid(\omega,T^{j}x_{0})\in {\rm Proj}_{\mathbb{T}^{d+1}}S_N, \ \exists|j|\sim N_2\}.
\end{equation}
By (\ref{d17}), (\ref{d18}), (\ref{d19}), using Lemma \ref{l4.6}, ${\rm mes}\mathcal{B}_{N}<N_2^{-c},c>0$.
Let
\begin{equation}\label{d20}
\mathcal{B}=\bigcap_{N_{0}}\bigcup_{N>N_{0}}\mathcal{B}_{N},
\end{equation}
then ${\rm mes}\mathcal{B}=0$. We restrict $\omega\notin\mathcal{B}$.

If $\omega\notin\mathcal{B}_N$, we have for all $|j|\sim N_2, \ (\omega,T^{j}x_{0})\notin {\rm Proj}_{\mathbb{T}^{d+1}}S_N$,
by (\ref{d13}),
\begin{equation}\label{d21}
  |G_{[j-N,j+N]}(E,x_0)(m,n)|<e^{N^{1-}-\frac{1}{100}|m-n|\chi_{|m-n|>\frac{N}{10}}}.
\end{equation}
Let $\Lambda=\bigcup\limits_{\frac{1}{4}N_2<j<2N_2}[j-N,j+N]\supset[\frac{1}{4}N_2,2N_2]$, by Lemma \ref{l5.1}, we deduce from (\ref{d21}) that
\begin{equation}\label{d22}
  |G_{\Lambda}(E,x_0)(m,n)|<e^{-\frac{1}{200}|m-n|}, \quad |m-n|>\frac{N_2}{10},
\end{equation}
and therefore
\begin{equation}\label{d23}
 |\xi_j|<e^{-\frac{1}{1000}|j|},\quad  \frac{1}{2}N_2\leq j\leq N_2 .
\end{equation}

Since $\omega\notin\mathcal{B}$, by (\ref{d20}), there is some $N_0>0$, such that for all $N\geq N_0,\omega\notin\mathcal{B}_N$.
So, (\ref{d23}) holds for $j\in\bigcup\limits_{N\geq N_0}[\frac{1}{2}e^{(\log N)^{2}},e^{(\log N)^{2}}]=[\frac{1}{2}e^{(\log N_0)^{2}},\infty)$.
This proves (\ref{d6}) for $j>0$, similarly for $j<0$. Hence
Theorem \ref{t5.2} follows.
\end{proof}

\subsection*{Acknowledgment}
The authors are very grateful to Dr. Y. Shi for the valuable suggestions.
This paper was supported by  National Natural
Science Foundation of China (No. 11790272 and No. 11771093).

\end{document}